\documentclass[12pt]{amsart}
\usepackage{amssymb,enumerate}
\usepackage{amsmath}

\usepackage[colorlinks=true, pdfstartview=FitV, linkcolor=blue, citecolor=blue, urlcolor=blue]{hyperref}
\usepackage{bm}
\usepackage{fullpage}
\usepackage{mathrsfs}
\usepackage{comment}

\def \a {\alpha_n}

\newcommand{\HH}{\mathfrak{H}}

\newtheorem{theorem}{Theorem}[section]
\newtheorem{lemma}[theorem]{Lemma}
\newtheorem{proposition}[theorem]{Proposition}

\newtheorem{definition}[theorem]{Definition}

\numberwithin{equation}{section}

 \newcommand{\RR}{\mathbb{R}}
  \newcommand{\dint}{\displaystyle \int}
  
  \newcommand{\dlim}{\displaystyle \lim}

\begin{document}

\title{Limit theorems for singular  Skorohod integrals}

 \author[D. Bell]{Denis Bell} 
\address{University of North Florida, Department of Mathematics, Jacksonville, Florida, USA}
\email{dbell@unf.edu}

\author[R. Bola\~nos]{Raul Bola\~nos}  
\address{University of Kansas, Department of Mathematics, Lawrence, Kansas, USA}
\email{rbolanos@ku.edu}

\author[D. Nualart]{David Nualart} \thanks {D. Nualart is supported by NSF Grant DMS 1811181.}
\address{University of Kansas, Department of Mathematics, Lawrence, Kansas,USA}
\email{nualart@ku.edu}

\begin{abstract}
In this paper we prove  the convergence in distribution of sequences of It\^o and Skorohod integrals with integrands having singular asymptotic behavior. These sequences include stochastic convolutions and  generalize the example  $\sqrt n\int _0^1 t^n B_tdB_t$ first studied by  Peccati and Yor in 2004.
\end{abstract}

\maketitle



\section{Introduction}

The main objective of this paper is to study the limit in law  of sequences of random variables defined by Skorohod integrals 
\begin{equation} \label{eq1a}
F_n = \int_0^1 \phi_n(t)u_t\delta B^H_t. 
\end{equation}
Here $u_t$ is a continuous process,  $B^H$ is fractional Brownian motion with Hurst parameter $H$ lying in the range $(0,1)$, and $\phi_n$ is a sequence of deterministic kernels converging (in some sense) to a delta function based at 1 (hence the ``singular" in the title of the paper). We show, under suitable conditions on the $\phi_n$ and $u$, that the limit law of  the couple $(B^H, F_n)$ has the form $(B^H,cu_1Z)$, where $Z$ is a $N(0,1)$ random variable, independent of $B^H$. 

Our study of limit problems of this type  was motivated by the special case  $H=\frac 12$, $u_t=B_t^H$ and $\phi_n(t)=\sqrt{n} t^n$, introduced  
 in Proposition 2.1 of \cite{py1}, and studied in Proposition 18 of \cite{petamwii}, and Example 4.2 in \cite{NN}. Then, $B_t^H$ is standard Brownian motion and the integrals are of classical It\^o type.    Quantitative bounds for such integrals  in the case  $H>\frac 12$ have  been established by Nourdin, Nualart \& Peccati \cite{nu} using estimates derived from Malliavin calculus and, more recently, by Pratelli \& Rigo in \cite{pr} for $H\in (1/4,1)$, using more a elementary (but nonetheless intricate) argument. 
 
 In this article, we provide a new approach to this problem, valid for a more general class of integrands exhibiting singular asymptotic behavior at the right-hand endpoint. The approach is based on the following observation. The singular behavior of the kernels $\phi_n$ in (1.1) as $n\to\infty$ implies that the limit law of $F_n$ is determined by the behavior of integrals over arbitrarily time intervals   $[1-\delta, 1]$. This makes it possible to study the limit law via the more tractable sequence of random variables
\begin{equation} \label{eq1a}
G_n = u_{\alpha_n} \int_{\alpha_n}^1 \phi_n(t)\delta  B^H_t,
\end{equation}
where  $\alpha_n$ is a sequence of times chosen such that  $\alpha_n\uparrow 1$ at a carefully chosen  rate. We show that $F_n$ and $G_n$ have that same limit in $L^2$, and hence in law. Furthermore, the Skorohod integrals $I_n$ in (1.2) are Gaussian, as is the process $B^H$. It thus suffices to show that the sequence $I_n$ has a convergent variance and is asymptotically uncorrelated with $B^H$.

In Section \ref{3.1}, we implement this argument in the case where $B^H$ is standard Brownian motion  (denoted  here by $B$) and the process $u_t$ is progressively measurable. In this case the integrals are It\^o integrals  and the argument is technically easier. The basic result in this section is Theorem \ref{thm1}. As a special case of this theorem, we obtain  the limit in law of the  aforementioned sequence 
$$
\sqrt n\int _0^1 t^n B_tdB_t.
$$
Theorem \ref{thm1} is extended in Theorems \ref{thm2} and \ref{thmRn} to double, and multiple, integrals respectively.
In Section \ref{3.2} we discuss the  problem for stochastic convolutions.

In Section \ref{sec4}, we study the case of fractional Brownian motion ($H\ne 1/2$). Here it turns out to be more convenient to work with the approximating sequence
$$
G_n = \int_0^1 \phi_n(t)u_1 \delta  B^H_t.
$$
As is usual in this subject, the cases $H\in (1/2, 1)$ and $H\in (0,1/2)$ seem to require slightly different hypotheses and analyses, with the latter proving more involved. Analogues of Theorem \ref{thm1} are presented in Theorems \ref{thm6} and \ref{thm7} for these two cases.  The proof involves the use the divergence operator on Wiener space and thus has the flavor of Malliavin calculus.
As an example of these theorems, we obtain the main result of \cite{pr}   concerning the limit law of the sequence
$n^H\int_0^1 t^n B^H_tdB_t^H$, for $H$ in the range $(1/4,1)$.

\section{Preliminaries}

Fractional Brownian motion with Hurst parameter $H\in (0,1)$, $B^H=\{ B ^H_t, t \in [0,1] \}$ is a zero mean Gaussian process with a covariance function given by
\begin{align} \label{cov}
R_H(t,s): =E[B_t B_s]= \frac 12 \left( t^{2H} + s^{2H} - |t-s|^{2H} \right),
\end{align}
where $s,t  \in [0,1]$. The Hilbert space $\HH$ is defined as the closure of the space of step functions $\mathcal{E}$ on $[0,1]$ endowed with the scalar product
\[
\langle \mathbf{1}_{[0,t]} , \mathbf{1}_{[0,s]} \rangle_{\HH} =  R_H(t,s).
\]
Then the mapping $\mathbf{1}_{[0,t]} \to B_t^H$ can be extended to a linear isometry between $\HH$ and the Gaussian space $\mathcal{H}_1$ spanned by $B^H$. When $H=\frac 12$, $B^H$ is just  a standard Brownian motion and $\HH= L^2([0,1]^2)$. 

When $H\in (\frac 12,1)$, the inner product of two step functions $\phi, \psi \in \mathcal{E}$ can be expressed as
\[
\langle \phi, \psi \rangle_{\HH}= \alpha_H \int_0^1 \int_0^1 \phi(s) \psi(t) | t-s|^{2H-2} dsdt,
\]
where $\alpha_H= H(2H-1)$.  The space of measurable functions $\phi$ on $[0,1]$, such that
\[
\|   \phi \|^2_{|\HH|} := \alpha_H \int_0^1 \int_0^1 |\phi(s) | |\phi(t) || t-s|^{2H-2} dsdt < \infty,
\]
denoted by $| \HH |$, is a Banach space and we have the continuous embeddings $L^{\frac 1H} ([0,1]) \subset |\HH|  \subset  \HH$.

When $H \in (0,\frac{1}{2})$, the covariance of the fractional Brownian motion $B^H$ can be expressed as
\[
 R_H(t,s) = \int_0^{s \wedge t} K_H(s,u)K_H(t,u) du,
\]
where $K_H(t,s)$ is a square integrable kernel defined as
\[
  K_H(t,s) = d_H \left(\left(\frac{t}{s}\right)^{H-\frac{1}{2}} (t-s)^{H-\frac{1}{2}} - (H-\frac{1}{2}) s^{\frac{1}{2}-H} \int_s^t v^{H-\frac{3}{2}} (v-s)^{H-\frac{1}{2}} dv\right) \,,
\]
for $0 <s<t$, with $d_H$ being a constant depending on $H$. The kernel $K_H$ satisfies the following estimates
\begin{equation}\label{kh.est1}
	|K_H(t,s)| \leq c_H \left((t-s)^{H-\frac{1}{2}} + s^{H-\frac{1}{2}}\right) \,,
\end{equation}
and
\begin{equation}\label{kh.est2}
	\left|\frac{\partial K_H}{\partial t} (t,s)\right| \leq c'_H (t-s)^{H-\frac{3}{2}} \,,
\end{equation}
for all $s<t$ and for some constants $c_H, c'_H$. 
Define a  linear operator $K^*_{H}$ from $\mathcal{E}$ to $L^2([0,1])$ as follows
\begin{equation}\label{kstar}
  (K^*_{H}\phi)(s) = \left( K_H(1,s)\phi(s) + \int_s^1 (\phi(t) - \phi(s)) \frac{\partial K_H}{\partial t} (t,s) dt  \right).
\end{equation}
The operator $K^*_{H}$ can be extended to a linear isometry between the Hilbert space $\mathfrak{H}$ and $L^2([0,1])$,  that is, for any $\phi, \psi \in \HH$, we have
\begin{align}  \label{ecu1}
\langle \phi, \psi \rangle_{\HH} = \langle K_H^* \phi, K^*_H \psi \rangle_{L^2([0,1])}.
\end{align}
The space of H\"older continuous functions of order $\gamma >\frac 12-H$ is included in  $\HH$.

 Next, we introduce the derivative operator and its adjoint, the divergence. Consider a smooth and cylindrical random variable of the form $F= f(B^H_{t_1}, \dots, B^H_{t_d})$, where $f \in C_b^{\infty}(\mathbb{R}^{d })$ ($f$ and its partial derivatives are all bounded). We define its Malliavin derivative as the $\mathfrak{H}$-valued random variable  $DF$     given by 
$$
 D_s F = \sum_{i=1}^d \frac{\partial f}{\partial x_i} (B^H_{t_1},\dots, B^H_{t_d})\mathbf{1}_{[0, t_i]}(s).
 $$
 For any   real number $ p \geq 1$, we define the Sobolev space $\mathbb{D}^{1,p}$  as the closure of the space of smooth and cylindrical random variables with respect to the norm $\|\cdot\|_{1,p}$ given by
$$
\|F\|^p_{1,p} = \mathbb{E}(|F|^p) +  E \left(  \|D F\|^p_{ \HH } \right).
$$
Similarly, if $\mathbb{W}$ is a general Hilbert space, we can define the Sobolev space of $\mathbb{W}$-valued random variables $\mathbb{D}^{1,p}(\mathbb{W})$.

The adjoint of the Malliavin derivative operator $D$, denoted as $\delta$, is called the {\it divergence operator}. A random element $u$ belongs to the domain of $\delta$, denoted as ${\rm Dom} \, \delta$, if there exists a positive constant $c_u$ depending only on $u$ such that
$$
|E( \langle DF, u \rangle_{\mathfrak{H}}) | \leq c_u \|F\|_{L^2(\Omega)}
$$ 
for any $F \in \mathbb{D}^{1,2}$.  If $u \in {\rm Dom} \, \delta$, then the random variable $\delta(u)$ is defined by the duality relationship
\[
E\left(F \delta(u)\right) =E(\langle D F, u \rangle_{\mathfrak{H}}) \, ,
\] 
for any $F \in \mathbb{D}^{1,2}$.    We make use of the notation $\delta(u)=\int_0^\infty u_t \delta B^H_t$ and call $\delta(u)$ the {\it Skorohod integral} of $u$ with respect to the fractional Brownian motion $B^H$. The Skorohod integral satisfies the following isometry property for any
element $u\in \mathbb{D}^{1,2} (\HH) \subset {\rm Dom } \delta$:
\[
E( \delta (u) ^2 ) = E( \|u\|_{\HH} ^2) + E ( \langle Du, (Du)^* \rangle_{\HH\otimes \HH}),
\]
where $(Du)^*$ is the adjoint of  $Du$. As a consequence, we have 
\begin{equation} \label{fg2}
E( \delta (u) ^2 ) \le E( \|u\|_{\HH} ^2) + E  (\| Du \|^2 _{\HH\otimes \HH}).
\end{equation}
We will make use of the following result.
\begin{lemma}  \label{lem2.1}
Let $F\in \mathbb{D}^{1,2}$ and let $g \in \HH$. Then  the process $Fg$ belongs to the domain of $\delta$ and
\[
\int_0^1 F g_t \delta B^H_t = F \delta(g) + \langle DF, g \rangle_{\HH}.
\]
\end{lemma}
We refer  to \cite{nu} and the references therein for a more detailed account of the properties of the fractional Brownian motion and its associated Malliavin calculus (and to \cite{be} for an introduction to the latter subject).

We will make use of the following property of the Gamma function.
\begin{lemma}\label{lem1}
For any $a,b$ positive 
$$\dlim_{n\to \infty}\dfrac{\Gamma(n+a)n^{b-a}}{\Gamma(n+b)}=1. $$
\end{lemma}
\begin{proof}
This is a direct application of Stirling's formula.
\end{proof}

Throughout the paper we will make use of the notion of  stable convergence  provided in the next definition. Suppose that  the fractional Brownian motion $B^H$ is   defined in a probability space $(\Omega, \mathcal{F}, P)$, where $\mathcal{F}$ is the $P$-completion of the $\sigma$-field generated by  $B^H$.

\begin{definition} \label{d:stable} 
Fix $d\geq 1$. Let $F_{n}$ be a sequence of random variables with values in $\RR^{d}$, all defined on the probability space $(\Omega, \mathcal{F}, P)$. Let $F$ be a $\RR^d$-valued random variable defined on some extended probability space $(\Omega', \mathcal{F}', P')$.  We say that $F_{n}$ {\it converges  stably} to $F$,  if
\begin{equation} \label{bb2}
\underset{n \rightarrow \infty}{\rm lim}E\left[Ze^{i\left\langle \lambda, F_{n} \right\rangle_{\RR^{d}} } \right] = E'\left[Ze^{i\left\langle \lambda, F \right\rangle_{\RR^{d}} } \right]
\end{equation}
for every $\lambda \in \RR^{d}$ and every bounded $\mathcal{F}$--measurable random variable $Z$. 
\end{definition}

Condition  (\ref{bb2}) is equivalent to saying that the couple $(B^H, F_n)$ converges in law to $(B^H,F)$ in the space $C([0,\infty)) \times \RR$
(see, for instance,  \cite[Chapter 4]{JacSh}).

\section{Singular limits of sequences of  It\^o integrals}

Let $B=\{B_t, t \ge 0\}$ be a standard Brownian motion.  Denote by
$\mathcal{F}_t$ the natural filtration generated by $B$. 
In this section we will study the asymptotic behavior of two types of sequences of It\^o integrals.  First, we discuss a class  of integrals on $[0,1]$ that include a sequence of deterministic kernels $\phi_n$ converging to a delta function based at $1$. Secondly, we  apply our argument to stochastic convolutions with this type of asymptotic behavior.

\subsection{Stochastic integrals concentrating at $t=1$}  \label{3.1}
Consider a sequence of  bounded  nonnegative functions   $\phi_n(t)$ on $[0,1]$, that  satisfies the following conditions:

\smallskip
\noindent
{\bf (h1)}:  There is a sequence $\a\uparrow 1$ such that  
\[
\dlim_{n\to \infty}  \dint_{\a}^1 \phi_n^2(t)\, dt  =\dlim_{n\to \infty}  \dint_0^1 \phi_n^2(t)\, dt = L >0.
\]

\smallskip
\noindent
{\bf (h2)}:   For any $\delta \in [0,1)$,   $\sup_{0\le t\le \delta}  \phi_n(t) \rightarrow 0$ as $n\rightarrow \infty$.

\medskip The aim of this section is to study the asymptotic behavior of the sequence of  It\^o integrals
\begin{equation} \label{eq1a}
F_n:= \int_0^1 \phi_n(t) u_t\, dB_t, \kern 10 pt  n\ge 1,
\end{equation}
where $u=\{u_t, t\in [0,1]\}$ is  a  progressively measurable process  such that  $\int_0^1 E(u_t^2)\, dt< \infty$.

\begin{theorem} \label{thm1}  
 Suppose that the process $u$ is continuous in $L^2(\Omega)$  at $t=1$ and the sequence $\phi_n$ satisfies conditions {\bf(h1)} and {\bf (h2)}. 
Then, the sequence $F_n$  introduced in (\ref{eq1a})   converges  stably,  as $n\to\infty$ to $ \sqrt{L} u_1 Z $,
where $Z$ is a $N(0,1)$  random variable independent of the process $B$.
\end{theorem}

\begin{proof}
 Define
$$
G_n := u_{\a}  \int_{\a}^1  \phi_n(t) dB_t,
$$
where  $\alpha_n $ is the sequence appearing in condition {\bf (h1)}.
Then, as $n\rightarrow\infty$, 
\begin{equation} \label{eq1}
E[G_n^2] =  E[u_{\a}^2]  \int _{\a}^1 \phi_n^2(t)\, dt \to E(u_1^2)L .
\end{equation}
Moreover, as $n\rightarrow \infty$, 
 \begin{align}
 \nonumber  E[F_nG_n] & = E\Big [ \int_{\a} ^1 \phi_n(t)  u_t\, dB_t \int_{\a}^1 \phi_n(t) u_{\a} \, dB_t  \Big]  =  \int_{\a}^1 \phi_n^2(t) E(u_tu_{\a}) \, dt  \\  \label{eq3}
 &=  \int_{\a} ^1 \phi_n^2(t)  E\Big [ u_{\a}\left( u_t-u_{\a}\right) \Big ]\, dt +  \int_{\a}^1 \phi_n^2(t) E(u_{\a}^2)\, dt  \to E(u_1^2)L,
 \end{align}
 because $\int_{\a}^1 \phi_n^2(t) E(u_{\a}^2)\, dt\to E(u_1^2)L$ and 
 $$
 \left| \int_{\a} ^1 \phi_n^2(t)  E\Big [ u_{\a}\left( u_t-u_{\a}\right) \Big ]\, dt\right| \leq E(u_{\a}^2)^{1/2} \sup_{\a\leq t\leq 1} E((u_t-u_{\a})^2)^{1/2} \int_{\a}^1 \phi_n^2(t)\, dt \to 0 $$
 by the $L^2$-continuity of $u$ at $t=1$.  
On the other hand, for any fixed $\delta \in [0,1)$,
\begin{align}  \nonumber
E[F_n^2] &=  \int_0^1 \phi_n^2(t) E[u_t^2]dt  \\  
& =   \int_0^{\delta} \phi_n^2(t)E(u_t^2) \, dt +\int_{\delta}^1  \phi_n^2(t) E(u_{t}^2-u_1^2)\, dt + \int_{\delta}^1  \phi_n^2(t) E(u_1^2)\, dt.
\label{eq5}
\end{align}
As $n \rightarrow \infty$, the third term in (\ref{eq5})  has limit $E(u_1^2)L$  and the first 
term converges to zero   in view of   hypothesis {\bf (h2)},  because 
$$
\left| \int_0^{\delta} \phi_n^2(t)E(u_t^2) \, dt\right| \leq  \sup_{0\le t\le \delta}   \phi^2_n(t) \int_0^1 E(u_t^2) dt.
  $$
Moreover, the second converges to zero  as $\delta \uparrow 1$, uniformly in $n$, because we can write
$$
\left|\int_{\delta}^1  \phi_n^2(t) E(u_{t}^2-u_1^2)\, dt   \right|\leq \sup_{\delta\leq t\leq 1} E(| u_{t}^2-u_1^2| )\dint_{0}^1 \phi_n^2(t)\, dt.
 $$
 Therefore,
 \begin{equation} \label{eq5a}
  \lim_{n\rightarrow \infty} E[F_n^2]  =E(u_1^2)L.
 \end{equation}

Now, (\ref{eq1}), (\ref{eq3}) and  (\ref{eq5a})  imply that  $F_n - G_n \to 0$ in $L^2$ , and hence also in law. Finally, notice that the sequence
$(B,    \int_{\a}^1 \phi_n(t) dB_t)$ converges in law in the space $C([0,1]) \times \mathbb{R}$ to $(B, \sqrt{L}Z)$, where $Z$ is a $N(0,1)$ random variable independent  of $B$. This completes the proof.
 \end{proof}

 An example of a sequence of functions satisfying conditions  {\bf(h1)} and {\bf (h2)} with $L=\frac  12$ is 
$$
\phi_n(t)=\sqrt{n}t^n.
$$
Indeed  condition {\bf (h2)} holds trivially and condition  {\bf (h1)} holds taking, for instance, 
  $\a=1-\frac{\log n}{n}$, because $\a^n \sim \frac 1n$ and, therefore,  as $n\rightarrow \infty$, 
$$
n \int_{\a} ^1 t^{2n} dt  =  \frac { n(1-\a^{2n})}{2n+1} \rightarrow \frac 12.
$$
Thus we have proved the following.
\begin{proposition} \label{Prop1}  
 The sequence   of It\^o integrals
 $$
  \sqrt n\int_0^1 t^nB_tdB_t
 $$
 converges stably, as $n\to\infty$, to $ {1\over \sqrt 2}B_1Z$, where $Z$ is a $N(0,1)$ random variable independent of the process $B$.
 \end{proposition}

 \noindent
 \textbf{Remarks:} 
\begin{itemize}
\item[(i)]  We note that Proposition \ref{Prop1} was obtained by Nourdin, Nualart \& Peccati in [2, Proposition 3.7] as a corollary of a theorem proved by integration by parts on Wiener space.
\item[(ii)]
  If we assume that  $E(u_t^2)$  is bounded on $[0,1]$, then it is easy to show that we can remove condition  {\bf (h2)} in Theorem \ref{thm1}.
  \end{itemize}

 The next result is an extension of Theorem \ref{thm1} to the case of double stochastic It\^o integrals, which is proved by similar arguments.
  We need the following condition on the sequence $\phi_n$, which is stronger than {\bf (h2)}:
 
 \smallskip
\noindent
{\bf (h3)}:   For any $\delta \in [0,1)$,   $\left(\sup_{0\le t\le \delta}  \phi_n(t)  \right) \left(\sup_{0\le t\le 1}  \phi_n(t) \right) \rightarrow 0$ as $n\rightarrow \infty$.

\begin{theorem} \label{thm2}
Let $u=\{u_{s,t}, 0\le s\le t \le 1\}$ be a  two-parameter process satisfying the following conditions:
\begin{itemize}
\item[(i)] $u_{s,t}$ is $\mathcal{F}_s $-measurable for $s\leq t$.
\item[(ii)] $\int_0^1 \int_0^t E(u_{s,t}^2)\, ds\, dt<\infty$.
\item[(iii)] $u_{s,t}$ is continuous at $(1,1)$ in the $L^2(\Omega)$ sense.  
\end{itemize}
 Consider the sequence of  iterated It\^o integrals
$$
F_n:=  \int_0^1 \int_0^t  \phi_n(s) \phi_n(t) u_{s,t}  dB_s\, dB_t, \kern 10 pt  n\ge 1,
$$
where the sequence $\phi_n$ satisfies conditions {\bf (h1)} and {\bf (h3)}.
Then $F_n$ converges  stably,  as $n\to\infty$ to  $\frac 12 L u_{1,1} H_2(Z) $,
where $Z\sim N(0,1)$ is independent of the process $B$, and $H_2=x^2-1$ is the second Hermite polynomial.
\end{theorem}

\begin{proof}
Define
$$
G_n := u_{\a,\a}  \int_{\a}^1 \int_{\a}^t  \phi_n(s) \phi_n(t) dB_s\, dB_t,
$$
where $\a$ is the sequence appearing in condition { \bf (h1)}.
 By   { \bf (h1)} we   have, as $n\rightarrow \infty$,
\begin{align}\label{eq6}
 E[G_n^2] =  E[u_{\a,\a}^2]  \int _{\a}^1  \int _{\a}^t  \phi^2_n(s) \phi^2_n(t) ds\, dt  \to  \frac {L^2}2 E(u_{1,1}^2).
\end{align}
Also
\begin{align*}
E[F_n^2]  =  \int_0^1 \int_0^t   \phi^2_n(s) \phi^2_n(t) E[u_{s,t}^2 ]\, ds\,dt.
\end{align*}
As in the proof of Theorem \ref{thm1}, fix $0<\delta<1$ and consider the decomposition
\begin{align} 
E[F_n^2] & =  \int_0^1   \int_0^{t\wedge \delta}    \phi^2_n(s) \phi^2_n(t) E[u_{s,t}^2  ]\, ds\,dt   \nonumber \\
& +  \int_{\delta}^{1} \int_{\delta}^t   \phi^2_n(s) \phi^2_n(t)  E[u_{s,t}^2 -u_{1,1}^2 ]\, ds\,dt +  E[u_{1,1}^2]\int_{\delta}^{1} \int_{\delta}^t   \phi^2_n(s) \phi^2_n(t) dsdt. \label{eq7}
\end{align}
The first  term of (\ref{eq7}) converges to zero as $n\rightarrow \infty$ by condition {\bf (h3)} and the third term converges to
$\frac {L^2}2 E(u_{1,1}^2)$. For the second term we have the estimate
\[
\left |\int_{\delta}^{1} \int_{\delta}^t   \phi^2_n(s) \phi^2_n(t)  E[u_{s,t}^2 -u_{1,1}^2 ]\, ds\,dt  \right| \le \sup_{\delta \le s \le t\le 1} |E[u_{s,t}^2 -u_{1,1}^2 ]|
\int_{0}^{1} \int_{0}^t   \phi^2_n(s) \phi^2_n(t) dsdt,
\]
which shows that this term converges to zero as $\delta \uparrow 1$, due to the continuity of $u$ and $(1,1)$, uniformly in $n$.
In this way, we obtain
\begin{align}
\lim_{n \to \infty} E(F_n^2)= \frac {L^2}2 E(u_{1,1}^2 ) \label{eq8}.
\end{align}
Also
 \begin{align}
 E[F_nG_n] & =  \int_{\a} ^1 \int_{\a} ^t  \phi^2_n(s) \phi^2_n(t) E(u_{s,t}u_{\a,\a})\, ds\, dt \nonumber \\
 & = \int_{\a} ^1 \int_{\a} ^t  \phi^2_n(s) \phi^2_n(t)  E\Big [ \left( u_{s,t} -u_{\a,\a}\right)u_{\a,\a} \Big ] \, ds\, dt \nonumber \\
& +  \int_{\a} ^1 \int_{\a} ^t   \phi^2_n(s) \phi^2_n(t)  E\Big [ u_{\a,\a}^2  \Big ] \, ds\, dt. \label{eq9}
\end{align}
At this point similar calculations to (\ref{eq6}) show that the second term in (\ref{eq9}) has limit $  \frac {L^2}2 E(u_{1,1}^2)$, while the first term of (\ref{eq9}) converges to $0$ due again to Cauchy-Schwartz  inequality, condition { \bf (h1)} and the continuity of $u$ at $(1,1)$ in $L^2$.
Consequently, as $n\rightarrow \infty$,
\begin{align}
E[F_nG_n] \to  \frac { L^2}2 E(u_{1,1}^2 ). \label{eq10}
\end{align}  
Thus (\ref{eq6}), (\ref{eq8}) and  (\ref{eq10})  imply that $F_n - G_n \to 0$ in $L^2$ , and hence also in law. Finally to see that the limit of $G_n$ has the desired form note that
$$
 \int_{\a}^1 \int_{\a}^t  \phi_n(s) \phi_n(t)\, dB_s\, dB_t = \frac{1}{2} I_2( \phi_n^{\otimes 2} \mathbf{1}_{[\a,1]^2}  ),
$$
where $I_2$ denotes the double It\^o-Wiener integral.
 Then by using the fact that multiple stochastic integrals of this form can be written in terms of Hermite polynomials, we can write
$$
G_n=\dfrac{1}{2} u_{\a,\a} \| \phi_n  \mathbf{1}_{[\a,1]}  \|_{L^2}^2 H_2\left( \dfrac{\dint_{\a}^1 \phi_n(t) dB_t}{\| \phi_n\mathbf{1}_{[\a,1]}\|_{L^2([0,1])}}  \right).
$$
Then the conclusion follows because $\| \phi_n  \mathbf{1}_{[\a,1]} \|_{L^2([0,1])}\to  \sqrt{L}$ as $n\rightarrow \infty$, $u_{\a,\a}$ converges to $u_{1,1}$ in $L^2(\Omega)$ as $n\rightarrow \infty$,  and  the sequence
$(B,   \int_{\a}^1  \phi_n(t) dB_t)$ converges in law in the space $C([0,1]) \times \mathbb{R}$ to $(B,  \sqrt{ L} Z)$, where $Z$ is a $N(0,1)$ random variable independent  of $B$.
 This implies that the limit law of $G_n$ has the stated form and completes the proof.
\end{proof}

\noindent
 {\bf Remarks:}  
 \begin{itemize}
 \item[(i)]  The previous theorem applies to the particular case $\phi_n(t)= \sqrt {n} t^n$, as before.
 
 \item[(ii)]
 One can consider the more general situation of a sequence of bounded symmetric functions $\phi_n(s,t)$ on  $[0,1]^2$, satisfying the following conditions:
 
 \smallskip
\noindent
{\bf (h12)}:  There is a sequence $\a\uparrow 1$ such that  
\[
\dlim_{n\to \infty}  \dint_{\a}^1 \dint_{\a}^1   \phi_n^2(s,t)dsdt  =\dlim_{n\to \infty}  \dint_0^1 \dint_0^1 \phi_n^2(s,t) dsdt= L >0.
\]

\smallskip
\noindent
{\bf (h22)}:   For any $\delta \in [0,1)$,   $\sup_{0\le s\le \delta, 0\le t\le 1} | \phi_n(s,t)| \rightarrow 0$ as $n\rightarrow \infty$.

\noindent
In  this case we need to compute the limit in law of $ I_2( \phi_n \mathbf{1}_{[\a,1]^2} )$,
which is a more complicated problem that requires additional conditions on the sequence $\phi_n$. We will not  treat this problem here.
\end{itemize}

Theorem \ref{thm2} can be extended to higher dimensions. The proof is
similar and omitted.   We need the following condition on the sequence $\phi_n$, which is stronger than {\bf (h2)}:
 
 \smallskip
\noindent
{\bf (h3m)}:   For any $\delta \in [0,1)$,   $\left(\sup_{0\le t\le \delta}  \phi_n(t)  \right) \left(\sup_{0\le t\le 1}  \phi_n(t) \right)^{m-1} \rightarrow 0$ as $n\rightarrow \infty$, where $m$ is the number of parameters.

\begin{theorem} \label{thmRn} 
Let $u= \{u_{t_1,\dots, t_m},  0\le t_1\le \cdots \le t_m  \le 1\}$ be an $m$-parameter stochastic process satisfying the following properties. 
 \vskip 4pt
\begin{itemize}
\item[i)] $u_{t_1,\dots,t_m} $ is $\mathcal{F}_{t_1}$-measurable.  
 \vskip 4pt
\item[ii)] $\displaystyle \int_0^1 \int_0^{t_m} \cdots  \int_0^{t_2} E\left(   u^2_{t_1, \dots ,t_m} \right) dt_1 dt_2\cdots dt_m<\infty $.
 \vskip 4pt
\item[iii)] $ u_{t_1,\dots,t_m} $ is continuous at $(1, \dots, 1)\in \mathbb{R}^m$ in the $L^2(\Omega)$ sense.  
\end{itemize}
Consider the sequence of iterated It\^o integrals
$$
F_n:=  \int_0^1 \int_0^{t_m} \cdots  \int_0^{t_2}  \phi_n(t_1) \cdots \phi_n(t_m) u_{t_1,\dots,t_m}  dB_{t_1} \cdots dB_{t_{m-1}}dB_{t_m}, \kern 10 pt  n\ge 1, $$
where the sequence $\phi_n$ satisfies conditions {\bf (h1)} and {\bf (h3m)}.
Then,  $F_n$ converges  stably, as $n\to\infty$, to  $({m! } )^{-1}L^{\frac m 2} u_{1,\dots,1} H_m(Z) $,
where $Z\sim N(0,1)$ is independent of the process $B$ and $H_m$ is the $m$th Hermite polynomial.
\end{theorem}

\subsection{Asymptotic behavior of stochastic convolutions} \label{3.2}
As before, let $B=\{B_t, t \ge 0\}$ be a standard Brownian motion and set $H= L^2([0,\infty))$.
 Consider a nonnegative continuous and bounded  function $\psi(x)$ on $\mathbb{R}$, such that $ \int_{-\infty} ^\infty \psi^2(x) dx =1$. Let $\psi_n(x) =\sqrt{n} \psi(nx)$.  Then $\psi^2_n$ is an approximation of the identity.

Let $u=\{u_t, t  \ge 0\}$ be an adapted and square integrable process. Define the stochastic convolution
\[
(u *_B \psi_n)_t=\int_0 ^\infty u_{s} \psi_n(t-s) dB_s, \quad t\ge 0.
\]
In this subsection we are interested in the asymptotic behavior of $(u *_B \psi_n)$ as $n$ tends to infinity.  The limit in law will have the form $u_t Z_t$, where $Z$ is a Gaussian process independent of $B$.

The following theorem is the main result of this subsection.

\begin{theorem}\label{MT}
Assume $u=\{ u_t,t\ge 0\}$ is an adapted, square integrable process, continuous at a fixed   time $t\ge 0$ in the  $L^2(\Omega)$ sense.
Consider a nonnegative continuous and bounded function $\psi(x)$ on $\mathbb{R}$, such that $ \int_{-\infty} ^\infty \psi^2(x) dx =1$  and $\psi^2(x)=o(|x|^{-1})$ as $x\to \infty$.
Then,   the stochastic convolution
 $ (u *_B \psi_n)_t$  converges stably to $u_t Z$  as $n\rightarrow \infty$, where $Z$ is a standard Gaussian random variable independent of $B$.   \end{theorem}

\begin{proof}
Let $\a$ be a sequence decreasing to $0$ so that $n\a\to \infty$. For $t\ge 0$, set 
$$
G_{n}=u_{(t-\a)_+} S_n,
$$ 
where $S_n=\dint_{R_n(t)}\psi_n(t-s)\, dB_s$ with $R_n(t)=\{s\ge 0:|t-s|\leq \a \}$.  Then  we can write
$$
E(S_n^2)=\dint_{R_n(t)}\psi_n^2(t-s)\, ds  =\dint_{|r|\leq \a, r\le t}\psi_n^2(r)\, dr = \dint_{|z|\leq n\a, z\le nt} \psi^2(z)\, dz \to 1 \ \ \ \text{as $n\to \infty$}.
$$
Moreover, since $u_{(t-\a)_+}$ is $\mathcal{F}_{(t-\a)_+}$ measurable we can write
$$
G_{n} =  \dint_{R_n(t)}u_{(t-\a)_+} \psi_n(t-s)\, dB_s
 $$
and therefore
$$
 E(G_{n}^2)= \dint_{R_n(t)} E(u_{(t-\a)_+}^2) \psi_n^2(t-s)\, ds = E(u_{(t-\a)_+}^2) \dint_{|s|\leq\a ,s\le t}  \psi_n^2(s)\, ds \to E(u_t^2).
 $$
On the other hand, by  Ito's isometry  property we can write
$$
 E\left((u *_B \psi_n)_t^2 \right)=\dint_0^\infty E(u_s^2)\psi_n^2(t-s)\, ds.
 $$
 That means,  $ E\left((u *_B \psi_n)_t^2 \right)$ is the convolution of $s\to E(u_s^2)$ with $\psi^2_n$, and by  Theorem 9.9 in \cite{WZ}, we deduce
$$
\lim_{n\to \infty} E\left[(u *_B \psi_n)_t^2\right]= E(u_t^2).
$$
Finally, by  It\^o's isometry and the $L^2$-continuity of $u$ at $t$ 
\begin{align*}
E\left[ (u *_B \psi_n)_t G_{n} \right] &= \dint_{R_n(t)} E(u_s u_{(t-\a)_+}) \psi_n^2(t-s)\, ds \\
&= \dint_{R_n(t)} E( u_{(t-\a)_+} (u_s - u_{(t-\a)_+})) \psi_n^2(t-s)\, ds \\
&+  E( u_{(t-\a)_+}^2) \dint_{R_n(t)}  \psi_n^2(t-s)\, ds \rightarrow E(u_t^2),
\end{align*} 
as $n\rightarrow \infty$.
Thus $(u *_B \psi_n)_t-G_{n}\xrightarrow{L^2(\Omega)} 0$ as $n\to \infty$ and hence in law. Finally, note that for each $n$, $u_{t-\a}$ and $S_n$ are independent random variables such that $u_{t-\a}$ converges to $u_t$  and $ 
{\rm Var} (S_n^2 ) \to 1$.\  This implies that the limit law of $G_n$ has the stated form and completes the proof.
\end{proof}

As in the proof of  Theorem \ref{MT}, if $\a$ is a sequence decreasing to $0$ so that $n\a\to \infty$, we can consider for each $t\geq 0$ the sequence of random variables
\begin{equation} \label{bb1}
 S_n^{t}:= \int_{|t-r|\leq \a} \psi_n(t-r) dB_r.
 \end{equation}
 The next lemma establishes the asymptotic behavior of the sequence of processes $ \{S^t_n, t\ge 0\}$.
 
 \begin{lemma}\label{lem1} The finite-dimensional distributions of the process $ \{S^t_n, t\ge 0\}$ introduced in (\ref{bb1}) converge stably  to those of a centered Gaussian process $\{Z_t, t\ge 0 \}$ independent of $B$ and with covariance function given by
  \begin{equation} \label{cov}
 E(Z_tZ_s)=
\begin{cases}
  1 & \quad {\rm if}  \,\, s=t\\
  0 & \quad {\rm if}  \,\, s\not=t.
 \end{cases}
 \end{equation}
\end{lemma}

\begin{proof}
Let $0\le t_1< t_2< \cdots< t_k$. We need to prove the convergence in law
$$
  (B, S_n^{t_1}, \dots, S_n^{t_k}) \xrightarrow{Law}   (B, Z_{t_1},\dots,Z_{t_k})
$$
in the space $C(\RR_+) \times \RR^k$. 
We can choose $N$ large enough so that for $n\ge N$,  the Gaussian random variable $S_n^{t_i}$ become uncorrelated and hence independent. Then 
as in the proof of Theorem \ref{MT},  it holds that
$$
 (S_n^{t_1}, \dots,S_n^{t_k})\xrightarrow{Law} (Z_{t_1}, \dots ,Z_{t_k}),
 $$
 where  the random vector $(Z_{t_1}, \dots,  Z_{t_k})$ has a standard Gaussian distribution on $\mathbb{R}^k$ and is   independent of $B$.
 This completes the proof. 
 \end{proof}

Notice that we cannot expect that the convergence in Proposition  \ref{prop1} holds in $C(0,\infty)$. Indeed, although under some mild conditions the stochastic convolution has a continuous version,  the process $Z$  does not have a continuous version.
 
The following proposition establishes the convergence of the stochastic convolution as a process in the sense of the finite-dimensional distributions.   
 
\begin{proposition} \label{prop1}
Under the assumptions of Theorem \ref{MT}, suppose that the process $u$ is continuous in $[0,\infty)$ in the $L^2$ sense. Then the finite-dimensional distributions of the process $\{ (X *_B \psi_n)_t, t\ge0\}$ converges stably to those of  $\{u_tZ_t, t\ge0\}$,
where $\{Z_t, t\ge 0 \}$ is a Gaussian process independent of $B$ with covariance function given by (\ref{cov}).
\end{proposition}

\begin{proof}
Let $0< t_1< t_2< \cdots< t_k$. We want to show that
\begin{equation} \label{ecua1}
(B, (u *_B \psi_n)_{t_1},(u *_B \psi_n)_{t_2}),\dots , (u *_B \psi_n)_{t_k})\xrightarrow{Law} (B,u_{t_1}Z_{t_1},u_{t_2}Z_{t_2},\dots ,u_{t_1}Z_{t_k}),
\end{equation}
where  the random vector $(Z_{t_1}, \dots,  Z_{t_k})$ has a standard Gaussian distribution on $\mathbb{R}^k$ and is   independent of $B$.
As in the proof of  Theorem \ref{MT}, if $\a$ is a sequence decreasing to $0$ such that $n\a\to \infty$, we can consider for each $t\geq 0$ the sequence of random variables $S_n^t$ defined in  (\ref{bb1}). 
Then, we have that,  by the proof of theorem \ref{MT}, for each $i=1,\dots,k$, 
$$ 
(u *_B \psi_n)_{t_i}-u_{(t_i-\a)_+ }S_n^{t_i} \xrightarrow{L^2} 0.
 $$
Also, by the $L^2$-continuity of $u$ and the Cauchy-Schwartz  inequality, we can write
$$
 u_{(t_i-\a})_+S_n^{t_i}-u_{t_i}S_n^{t_i }\xrightarrow{L^1} 0.
 $$
In particular the above convergence holds also in probability, so that 
$$
A_n^{t_i}:= (u *_B \psi_n)_{t_i}-u_{t_i}S_n^{t_i } \xrightarrow{P} 0 
$$
for $i=1,\dots,k$.
As a consequence,
$$
 (A_n^{t_1},A_n^{t_2},\dots,A_n^{t_k})\xrightarrow{P}(0,0,\dots,0).
 $$
Then by Slutsky's theorem   (\ref{ecua1}) follows from    the convergence in law
$$
  (B, u_{t_1}S_n^{t_1}, \dots,u_{t_k}S_n^{t_k}) \xrightarrow{Law}   (B, u_{t_1}Z_{t_1},\dots,u_{t_k}Z_{t_k}),
$$
which is a consequence of Lemma \ref{lem1}.
 This completes the proof. 
\end{proof}

\section{Skorohod integrals  with respect to fractional Brownian Motion} \label{sec4}

Consider a fractional Brownian motion $B^H=\{ B^H_t, t\in [0,1]\}$ with Hurst parameter $H \in (0,1)$. That is, $B^H$ is a zero mean Gaussian process with covariance function (\ref{cov}).
In this section we will study the asymptotic behavior as $n\rightarrow \infty$  of a sequence of Skorohod integrals of the form
\begin{equation}   \label{bb3}
F_n=  \int_0^1  \phi_n(t) u_t \, \delta B_t^H, \kern 10 pt  n\ge 1,
\end{equation}
where $u$ is a stochastic process verifying some suitable conditions. We split our study in two cases according to whether $H>1/2$ or $H<1/2$.

\subsection*{Case $H>1/2$}
We will assume the following conditions on the sequence $\phi_n$ of nonnegative and bounded functions:

\smallskip
\noindent
{\bf (h4)}:  $ \lim_{n \rightarrow \infty} \| \phi_n\|_{\HH} ^2=L>0$.

\smallskip
\noindent
{\bf (h5)}:  $ \lim_{n \rightarrow \infty} \| \phi_n\|_{r} =0$ for some $r<\frac 1H$ (where here, and in the sequel. $||\cdot||_r$ denotes the $L^r$-norm on $[0,1]$).

\medskip
 We are now ready to state and prove the main results of this section.
 
\begin{theorem} \label{thm6}
Assume $B^H$ is a fractional Brownian motion with Hurst parameter $H>1/2$. 
Consider a sequence of nonnegative and bounded functions $\phi_n$ on $[0,1]$ satisfying conditions  {\bf (h3)},  {\bf (h4)} and {\bf (h5)}.  Let $u$ be a stochastic process satisfying the following conditions:
\begin{itemize}
\item[(i)]  For any $t\in [0,1]$,  $u_t\in \mathbb{D}^{1,2} $ and the mapping $t\to \|u_t \|_{1,2}$ belongs to $\HH$.
\item[(ii)]  $u_t$ is continuous in $\mathbb{D}^{1,2}$ at $t=1$. 
\item[(iii)] $ \int_0^1 (E[ | D_s u_1 |] )^p ds <\infty$ where  $\frac 1 p+\frac 1r =2H$, and $r$ is the number appearing in condition {\bf (h5)}. 
\end{itemize}
Consider the sequence of Skorohod integrals introduced in (\ref{bb3}).
Then $F_n$ converges stably as $n\to\infty$ to $ u_1 \sqrt{L }Z $,
where $Z$ is a  $N(0,1)$ random variable independent of $B^H$.
\end{theorem}

\begin{proof}
Notice first that conditions (i) and (ii) imply that $\phi_n(t) u_t$ belongs to $\mathbb{D}^{1,2} (\HH) \subset {\rm Dom} \delta$. 
Set $G_n :=  \displaystyle \int_{0}^1 \phi_n(t) u_1 \delta B_t^H $. Denoting $\alpha_H=H(2H-1)$, in view of   (\ref{fg2}) we can write
\begin{align}
E\left[ (F_n-G_n)^2\right] & \leq  E( \| \phi_n(t)(u_t-u_1) \|_{\HH}^2) +   E (\| \phi_n(t) D(u_t-u_1) \|_{\HH \otimes \HH }^2)  \nonumber \\
& = \alpha_H\, \int_0^1 \, \int_0^1 \phi_n (t)\phi_n(s)  E\Big [ (u_t-u_1)(u_s-u_1 \Big ] \vert t-s\vert^{2H-2}\ ds dt  \nonumber \\
&  \quad + \alpha_H\, \int_0^1 \, \int_0^1  \phi_n (t)\phi_n(s) E\Big [ \langle D (u_t-u_1), D(u_s-u_1)\rangle_{\HH} \Big ] \vert t-s\vert^{2H-2} ds dt  \nonumber \\
& =A_{1,n} + A_{2,n}. \label{eq11} 
\end{align}
Both terms in \eqref{eq11} are handled similarly and we will show the details only for the second one. Let $0<\delta <1$. Then, separating the second term in two integrals, yields
\begin{align}
A_{2,n} &=  \alpha_H \int_0^{1}  \int_0^1  \mathbf{1}_{\{s \wedge t \le \delta\}} \phi_n (t)\phi_n(s) E\Big [ \langle D(u_t-u_1) ,D(u_s-u_1) \rangle_{\HH} \Big ] \vert t-s\vert^{2H-2}  dsdt
\nonumber \\
&  \quad +  \alpha_H \int_{\delta}^1  \int_{\delta}^1 \phi_n (t)\phi_n(s) E\Big [ \langle D(u_t-u_1), D(u_s-u_1) \rangle_{\HH} \Big ] \vert t-s\vert^{2H-2} dsdt. \label{eq12} 
\end{align}
At this step note that by condition (i) 
\begin{align*}
 &\int_0^1 \int_0^1 | E\left( \langle D(u_t-u_1), D(u_s-u_1)\rangle_{\HH} \right)| \vert t-s\vert^{2H-2} dsdt    \\
 &  \quad \le \int_0^{1} \int_0^{1}  \| u_t-u_1 \|_{1,2} \| u_s-u_1\|_{1,2} \vert t-s\vert^{2H-2} dsdt   <\infty. 
 \end{align*}
So there is a constant $C$ such that the first  term in (\ref{eq12}) is bounded by
$$
C  \sup_{s\wedge t\le \delta} \phi_n(s) \phi_n(t),
$$
 which converges to $0$ as $n\rightarrow \infty$ by condition {\bf (h3)}.

On the other hand, for the second term in (\ref{eq12}), it follows from Cauchy-Schwartz inequality  that
\begin{align*}
&   \int_{\delta}^1  \int_{\delta}^1 \phi_n (t)\phi_n(s) E\Big [ \langle D(u_t-u_1), D(u_s-u_1) \rangle_{\HH} \Big ] \vert t-s\vert^{2H-2} dsdt\\
& \leq \sup_{t \in [\delta,1]} E\Big [ \| D(u_t-u_1)\|_{\HH}^2 \Big ]   \left(   \int_{0}^{1} \int_{0}^1  \phi_n(s) \phi_n(t)  \vert t-s\vert^{2H-2}  ds\,dt \right).
\end{align*}
By condition   {\bf (h4)}, the sequence $\int_{0}^{1} \int_{0}^1  \phi_n(s) \phi_n(t)  \vert t-s\vert^{2H-2}  ds\,dt  $ is bounded and by condition (ii)  the first factor tends to zero as $\delta \rightarrow 1$.  This shows that $A_{2,n}$ tends to zero as $n\rightarrow \infty$. 
Repeating the same argument, we obtain that $A_{1,n}$ tends to zero as $n\rightarrow \infty$.

We have shown that $F_n - G_n \to 0$ in $L^2$, and hence also in law. All is left is to show that the limit of $G_n$ has the desired form. To this end note that, applying Lemma \ref{lem2.1}, $G_n$ can also be written as
\begin{align}
 G_n&= u_1  \dint_0^1  \phi_n(t)  \delta B_t^H + \alpha_H   \dint_0^1 \dint_0^1 \vert t-s\vert^{2H-2}\,  \phi_n(t)\, D_s u_1 ds  dt \nonumber \\
 & =: u_1 B_{1,n} + \alpha_H B_{2,n}.   \label{eq13}
\end{align}
Let $p$ be as in the statement of the theorem and note that $p>1$. Applying  H\"older's inequality with $\frac 1p +\frac 1q =1$, yields
\begin{align*}
E[ |B_{2,n}|]  & \le    \left( \int_0^1 (E[ | D_s u_1 |] )^p ds \right)^{\frac 1p}  \left(  \int_0^1 
 \left(\int_0^1 \vert t-s\vert^{2H-2} \phi_n(t) dt \right)^q ds  \right)^{\frac 1q}.
 \end{align*}
The second factor is the  $L^q$-norm of the fractional integral of order $2H-1$ of the function $\phi_n$ on $[0,1]$.  By the  Hardy-Littlewood inequality, this factor is bounded by a constant times $ \| \phi_n \|_{L^r([0,1])}$, where $\frac1 r = \frac 1q +2H-1=2H -\frac 1p$. Taking into account conditions (iii) and {\bf (h5)}, we deduce
\[
\lim_{n\rightarrow \infty} E[ |B_{2,n}|] =0.
\]
In order to complete the proof of the theorem it suffices to show that $(B^H,  B_{1,n})$ converges in law in the space $C([0,1]) \times \mathbb{R})$ to $(B^H, \sqrt{ L} Z)$, where $Z$ is a $N(0,1)$ random variable independent of $B^H$.
In view of the fact that $B^H$ is a Gaussian process, this will follow from the next  two properties:

\medskip
\noindent (a): $\lim_{n \rightarrow \infty} E[ B_{1,n} ^2] =L$, which follows from  property {\bf (h4)}.

\medskip
\noindent (b):  For any $t_0\in [0,1]$, $\lim_{n \rightarrow \infty} E[ B_{1,n}  B^H_{t_0} ]=0$. In fact,   using property {\bf (h5)}, we obtain
for $\frac 1r +\frac 1{r'} =1$,
\begin{align*}
  E[ B_{1,n}  B^H_{t_0}] &= \alpha_H \dint_0^1 \int_0^{t_0}  \phi_n(t)  \vert t-s\vert^{2H-2}\, ds\, dt  \\
  &\le  \alpha_n\| \phi_n \| _r   \left(  \int_0^1 \left( \int_0^{t_0} |t-s|^{2H-2} ds \right)^{r'} \right)^{1/r'} \to 0,
\end{align*}
as $n \rightarrow \infty$. 
\end{proof}

Theorem  \ref{thm6} can be applied to the example  $\phi_n(t) = n ^H t^n$, and in this case, $L= H \Gamma(2H)$.
Indeed, condition {\bf (h3)} is obvious. Condition {\bf (h4)} follows from  Lemma \ref{lem2} below.
Condition {\bf (h5)} holds for any $r<\frac 1H$. This means that in condition (iii) it suffices to show that the integral is bounded for some $p> \frac 1H$. 

\begin{lemma}\label{lem2}
For any $n,m \in \mathbb{N}$ and $r>-1$
\begin{align*} 
 \int_0^1 \int_0^{1} t^n s^m \vert t-s \vert^{r}\ ds\, dt = \dfrac{ \  \Gamma(m+1)\Gamma(r+1)}{(n+m+r+2)\Gamma(2+m+r)}+ \dfrac{ \  \Gamma(n+1)\Gamma(r+1)}{(n+m+r+2)\Gamma(2+n+r)}.
\end{align*}
In particular for $H>1/2$
$$\dlim_{n\to \infty}  n^{2H}\int_0^1 \int_0^{1} x^n y^n \vert x-y \vert^{2H-2}\ dy\, dx\ =\Gamma(2H-1). $$
\end{lemma}

\begin{proof}
First of all, note that using $y=zx$ yields
$$
\int_0^x  y^m ( x-y )^{r}\ dy=x^{m+1+r} \int_0^1 z^m(1-z)^{r}\, dz = x^{m+1+r} B(m+1,r+1),
 $$
 where $B$ denotes the Beta function.
Then  
\begin{align*}
& \dint_0^1 \int_0^{1} t^n s^m \vert t-s\vert^{r}\ ds\, dt  \\
& = \dint_0^1 \int_0^{t} t^n s^m (t-s)^{r}\ ds\, dt +  \int_0^1 \int_0^{s} t^n s^m (s-t)^{r}\ dt\, ds \\
& = \dint_0^1 t^{n+m+r+1}B(m+1,r+1)\, dt +\dint_0^1 s^{n+m+r+1}B(n+1,r+1)\\
&  = \dfrac{B(m+1,r+1)+B(n+1,r+1)}{n+m+r+2},
\end{align*}
The first part of the lemma now follows from the well-known relationship between the Beta and Gamma functions. The second part follows by taking $n=m$ and using  Lemma \ref{lem1}.
\end{proof}

%%%%%%%%%%%%%%%%%%%%%%%%%%%%%%%%%%%%%%%%%%%%%%%%%%%%%%%%%%%%%%%%%%%%%%%%%%%%%%%
%%%%%%%%%%%%%%%%%%%%%%%%%%%%%%%%%%%%%%%%%%%%%%%%%%%%%%%%%%%%%%%%%%%%%%%%%%%%%%
%%%%%%%%%%%%%%%%%%%%%%%%%%%%%%%%%%%%%%%%%%%%%%%%%%%%%%%%%%%%%%%%%%%%%%%%%%%%%%
%%%%%%%%%%%%%%%%  %%%%%%%%%%%%%%%%%%%%%%%%%%
%%%%%%%%%%%%%%%%%%%%%%%%%%%%%%%%%%%%%%%%%%%%%%%%%%%%%%%%%%%%%%%%%%%%%%%%%%%%%
%%%%%%%%%%%%%%%%%%%%%%%%%%%%%%%%%%%%%%%%%%%%%%%%%%%%%%%%%%%%%%%%%%%%%%%%%%%%

\subsection*{Case $H<1/2$} 
We assume the following conditions on the sequence $\phi_n$ of nonnegative and bounded functions:

\smallskip
\noindent
{\bf (h6)}:  $ \sup_n \dint_0^1 (s^{2H-1} + (1-s)^{2H-1}) \phi^2_n (s) ds <\infty$.   

\smallskip
\noindent
{\bf (h7)}:   For any $\delta \in [0,1)$, we have
\[
 \dlim_{n\rightarrow \infty}  \dint_0^{\delta} \left( \dint_s^{\delta}  |\phi_n(t) -\phi_n(s)|  (t-s)^{H-\frac 32} dt \right)^2 \, ds =0.
\]

\smallskip
\noindent
{\bf (h8)}:    $ \dlim_{n\rightarrow \infty}  \int_0^1 | (K_H^* \phi_n)(s)|^p ds =0$ for some $p>1$.

\begin{theorem} \label{thm7}
Assume $B^H$ is a fractional Brownian motion with Hurst parameter $0<H<1/2$. 
Consider a sequence of nonnegative and bounded functions $\phi_n$ on $[0,1]$ satisfying conditions  {\bf (h3)},  {\bf (h4)}, {\bf (h6)},  {\bf (h7)} and {\bf (h8)}.
Let $u$ be a stochastic process satisfying the following conditions:
\begin{itemize}
\item[(i)] For all $t\in [0,1]$,  $u_t\in \mathbb{D}^{1,2}$.
\item[(ii)]  The mapping $t\to u_t$ is H\"older continuous of order $\gamma > 1/2-H$ from $[0,1]$ into $\mathbb{D}^{1,2}$. 
\item[(iii)]    We have
\begin{align*}
\int_0^1    E(|  (K_H^*Du_1)(s)|^q)  ds  <\infty,
\end{align*}
where $\frac 1p + \frac 1q=1$ and $p$ is the exponent appearing in condition {\bf (h8)}. 
\end{itemize}
Consider the sequence of Skorohod integrals introduced in (\ref{bb3}).
Then $F_n$ converges  stably as $n\to\infty$ to  $ u_1 \sqrt{L} Z $,
where $Z$ is a $N(0,1)$ random  variable independent of $B^H$.
\end{theorem}

\begin{proof}

We divide the proof into 3 steps.

\medskip
\noindent
\textit{Step 1:} We need to compute the variance of  the random variable  $   \dint_0^1  \phi_n(t) \delta B_t^H$.
Condition {\bf (h4)} implies that
\[
\lim_{n\rightarrow \infty} E \left( \left|  \dint_0^1  \phi_n(t) \delta B_t^H \right|^2 \right) =L.
\]

  \medskip
\noindent
\textit{Step 2:}  Showing $F_n-G_n\xrightarrow{L^2(\Omega)}0$, where $$G_n :=  \displaystyle \int_{0}^1  \phi_n(t) u_1 \delta B_t^H. $$  
As in the proof of theorem \ref{thm6}, we can write
\begin{align*}
E\left[(F_n-G_n)^2\right] & \leq   E (\| \phi_n(t)(u_t-u_1) \|_{\mathcal{H}}^2) +  E( \| \phi_n(t) D(u_t-u_1) \|^2_{\mathcal{H} \otimes \mathcal{H} })=: C_{1,n} + C_{2,n}.\\
\end{align*}
We only work with $C_{2,n}$, the analysis of  $C_{1,n} $ being similar by changing  $\phi_n(t)D(u_t-u_1)$ and  $\HH$ appropriately by  $\phi_n(t) (u_t-u_1)$ and $\mathbb{R}$ in the argument below. We have, using (\ref{ecu1}) and (\ref{kstar}), 
\begin{align*}
 C_{2,n}
& =    E  \Bigg( \Big\| \dint_{0}^1   K_H(1,s) \phi_n(s) D(u_s-u_1)\\
&+\dint_s^1 \left(  \phi_n(t) D(u_t-u_1)-\phi_n(s) D(u_s-u_1)\right) \dfrac{\partial K_H} {\partial t} (t,s)  dt  \Big\|_{L^2([0,1];\HH)}^2 \Bigg).
\end{align*}
Since 
$$ (\phi_n(t) D(u_t-u_1)-\phi_n (s) D(u_s-u_1)) = \phi_n (s) D(u_t-u_s)+(\phi_n (t)-\phi_n (s))D(u_t-u_1),  $$
we obtain
\begin{align*}
 C_{2,n} & \leq  9 E \left(  \left\| K_H(1,s) \phi_n (s) D(u_s-u_1) \right\|_{L^2([0,1];\HH)}^2 \right) \\
& +    9   E\left(  \left\| \dint_s^1 \phi_n (s) D(u_t-u_s) \dfrac{\partial K_H} {\partial t} (t,s) dt \right\|_{L^2([0,1];\HH)}^2 \right)\\
& +  9E\left(  \left\| \dint_s^1 (\phi_n (t)-\phi_n (s))D(u_t-u_1)  \dfrac{\partial K_H} {\partial t} (t,s) dt\right\|_{L^2([0,1];\HH)}^2 \right)\\
&=: R_{1,n} + R_{2,n} + R_{3,n}.
\end{align*}
To handle the term  $R_{1,n}$ we note that, by (\ref{kh.est1}),     there is a constant $d_H$ such that
\begin{align} \label{ecu7}
    K(1,s)^2 \leq d_H (\, (1-s)^{2H-1}+s^{2H-1}).
\end{align} 
We will denote by $C$ a generic constant that may vary from line to line. 
Then by Minkowski's inequality and  condition (ii)   for any $\delta \in [0,1)$ we obtain
\begin{align*}
 R_{1,n} & \le9 E \left( \dint_0^1 K_H(1,s)^2  \phi_n^2(s) \|D(u_s-u_1)\|_{\HH}^2 \, ds \right)\\
 &\leq C    \dint_0^1  K_H(1,s)^2 \phi_n^2(s)   \|D(u_s-u_1)\|_{L^2[\Omega;\HH]}^2 \, ds \\
 &\le C    \dint_0^ \delta  K_H(1,s)^2 \phi_n^2(s)   (1-s)^{2\gamma}  ds\\
 &+C    \dint_\delta^1 K_H(1,s)^2 \phi_n^2(s)   (1-s)^{2\gamma}   ds\\
 &=: R_{12,n} + R_{22,n}.
 \end{align*}
 The term $R_{12,n}$ can be estimated as follows
 \[
 R_{12,n} \le C \sup_{0\le s\le \delta} \phi_n^2(s) \int_0^1K_H(1,s)^2 (1-s)^{2\gamma}   ds,
 \]
 Taking into account that  $ \int_0^1K_H(1,s)^2 (1-s)^{2\gamma}   ds <\infty$, we deduce from condition {\bf (h3)} that
 $ R_{12,n}$ converges to zero as $n\rightarrow \infty$. 
 For $ R_{22,n}$  we can write
 \[
 R_{22,n} \le C (1-\delta)^{2\gamma}  \int_0^1K_H(1,s)^2  \phi_n^2(s) ds.
 \]
 From (\ref{ecu7}) and condition {\bf (h6)}, we deduce that  $\sup_n  R_{22,n}  \rightarrow 0$ as $\delta \uparrow 1$. 
 Therefore, we have proved that
 \begin{align}  \label{ecu10}
 \lim_{n\rightarrow \infty}  R_{1,n} =0.
 \end{align} 
 Concerning the term  $R_{2,n}$, using Minkowski's inequality, the estimate  (\ref{kh.est2}) and condition (ii), we obtain
 \begin{align*}
 R_{2,n} & = 9   E\left(  \left\| \dint_s^1  \phi_n(s) D(u_t-u_s) \dfrac{\partial K_H} {\partial t} (t,s)  dt \right\|_{L^2([0,1];\HH)}^2 \right)\\  
& \leq C     \dint_0^1 \left( \dint_s^1 \phi_n(s) \|D(u_t-u_s)\|_{L^2(\Omega;\HH)} \left| \dfrac{\partial K_H} {\partial t} (t,s) \right| \, dt \right)^2 \, ds  \\
& \leq  C  \dint_0^1   \phi_n^2(s) \left( \dint_s^1  (t-s)^{\gamma}  (t-s)^{H-3/2} \, dt \right)^2 \, ds  \\
 & \leq  C  \dint_0^1  \phi_n^2(s)  (1-s)^{ 2\gamma + 2H-1} ds. 
  \end{align*}
Then,  for any $\delta \in [0,1)$, the integral  $\dint_0^ \delta  \phi_n^2(s)  (1-s)^{ 2\gamma + 2H-1} ds$ converges to zero as
$n\rightarrow \infty$ due to condition {\bf (h3)}, whereas, by condition  {\bf (h6)}, 
\[
\dint_\delta^ 1  \phi_n^2(s)  (1-s)^{ 2\gamma + 2H-1} ds \le (1-\delta)^{2\gamma} \int_0^1 \phi_n^2(s)  (1-s)^{   2H-1} ds 
\le C(1-\delta)^{2\gamma} \to 0,
\]
as $\delta \uparrow 1$. 
 Therefore, we have proved that
 \begin{align}  \label{ecu10}
 \lim_{n\rightarrow \infty}  R_{2,n} =0.
 \end{align} 
Finally for $R_{3,n}$ taking $0<\delta<1$, it follows from Minkowski's inequality that 
\begin{align*}
R_{3,n} &=9  E\left(  \left\| \dint_s^1 (\phi_n(t)-\phi_n(s))D(u_t-u_1)  \dfrac{\partial K_H} {\partial t} (t,s)dt\right\|_{L^2([0,1];\HH)}^2 \right) \\
&\leq  C \dint_0^1 \left( \dint_s^1 (\phi_n(t)-\phi_n(s))\|D(u_t-u_1)\|_{L^2(\Omega;\HH)} \left|  \dfrac{\partial K_H} {\partial t} (t,s)\right| \, dt \right)^2 \, ds  \\
  &\leq   C  \dint_0^{\delta} \left( \dint_s^{\delta}  (\phi_n(t)-\phi_n(s)) \|D(u_t-u_1)\|_{L^2(\Omega;\HH)} \left| \ \dfrac{\partial K_H} {\partial t} (t,s)\right| \, dt \right)^2 \, ds \\
    &+  C   \dint_0^{\delta} \left( \dint_{\delta}^1  (\phi_n(t)-\phi_n(s)) \|D(u_t-u_1)\|_{L^2(\Omega;\HH)}  \left|  \dfrac{\partial K_H} {\partial t} (t,s)\right|  \, dt \right)^2 \, ds  \\
    &+  C  \dint_{\delta}^1 \left( \dint_s^1 (\phi_n(t)-\phi_n(s)) \|D(u_t-u_1)\|_{L^2(\Omega;\HH)} \left|  \dfrac{\partial K_H} {\partial t} (t,s) \right|  \, dt \right)^2 \, ds\\
    &=: T_{1,n} + T_{2,n} + T_{3,n}. 
\end{align*}
  At this step we study each term separately. For both $T_{2,n}$ and  $T_{3,n}$ note that $\delta\leq t$ so condition (ii) gives 
  $$ 
  \|D(u_t-u_1)\|_{L^2(\Omega;\HH)}\leq  C (1-\delta)^{\gamma}.
  $$ 
  Also $\left|  \dfrac{\partial K_H} {\partial t} (t,s)\right|=-  \dfrac{\partial K_H} {\partial t} (t,s)$ so that  
  \begin{align*}
T_{2,n}+T_{3,n} & \leq  C(1-\delta)^{2\gamma}\dint_0^1 \left( \dint_s^1  (\phi_n(t)-\phi_n(s)) \left|  \dfrac{\partial K_H} {\partial t} (t,s)\right| \, dt\right)^2 ds \\
 & = C(1-\delta)^{2\gamma}\dint_0^1 \left( \dint_s^1  (\phi_n(t)-\phi_n(s)) \dfrac{\partial K_H} {\partial t} (t,s)dt\right)^2\, ds \\
  & \leq  2C(1-\delta)^{2\gamma}  \dint_0^1  \left( K_H(1,s) \phi_n(s) + \dint_s^1  (\phi_n(t)-\phi_n(s))  \dfrac{\partial K_H} {\partial t} (t,s)dt\right)^2\, ds \\
  & + 2C(1-\delta)^{2\gamma}  \dint_0^1  K_H^2(1,s) \phi^2_n(s) \, ds \\
  & = 2C(1-\delta)^{2\gamma} \|\phi_n\|^2_{\HH} + 2C(1-\delta)^{2\gamma}  \dint_0^1  K_H^2(1,s)\phi_n^2(s)\, ds.
  \end{align*}
  By condition {\bf (h4)}, $\|\phi_n\|^2_{\HH}$ is bounded uniformly in $n$, and by  condition {\bf (h6)},     $ \dint_0^1  K_H^2(1,s) \phi_n^2(s)  ds$ is bounded as well.  Therefore,
  \[
  \lim_{\delta \uparrow 1}\sup_n (T_{2,n}+T_{3,n})=0.
  \]   
  Thus,  in order to show that $C_{2,n}$ converges to zero as $n\rightarrow \infty$, it suffices to show that, for a fixed $\delta \in (0,1)$,
  \begin{align} \label{ecu11}
  \lim_{n\rightarrow \infty} T_{1,n} =0.
  \end{align}
Using Minkowski's inequality, condition (ii) and the estimate (\ref{kh.est2}), we can write
\begin{align*}
T_{1,n}& = C \dint_0^{\delta} \left( \dint_s^{\delta}  |\phi_n(t) -\phi_n(s)| \|D(u_t-u_1)\|_{L^2(\Omega;\HH)} \left|  \dfrac{\partial K_H} {\partial t} (t,s) \right| \, dt \right)^2 \, ds \\
& \leq   C \dint_0^{\delta} \left( \dint_s^{\delta}  |\phi_n(t) -\phi_n(s)|  (t-s)^{H-\frac 32} dt \right)^2 \, ds,
\end{align*}
which converges to $0$ as $n\rightarrow \infty$ by condition {\bf  (h7)}. 
This completes the proof of step 2.
 
\medskip
\noindent
\textit{Step 3:} We show that the limit in law of $G_n$ has the desired form. To this end note that $G_n$ can also be written as
\begin{align}
G_n=  u_1  \dint_0^1 \phi_n(t) \delta B_t^H + \langle  \phi_n ,D u_1\rangle_{\HH}.  \label{Var3}
\end{align}
First we will show using condition (iii) that 
\begin{align} 
 E(|  \langle  \phi_n ,D u_1\rangle_{\HH} |) \rightarrow 0  \label{fg1}
\end{align}
 as $n\rightarrow  \infty$.
 Fix $\delta \in [0,1)$. We can write, by H\"older's inequality 
\begin{align*}
E(|  \langle  \phi_n ,D u_1\rangle_{\HH} |) & =E(| \langle  (K^*_H\phi_n), (K^*_H Du_1) \rangle_{L^2([0,1])}|)\\
&\le  \left(    \int_0^1 | (K_H^* \phi_n)(s)|^p ds \right)^{\frac 1p}
\left(\int_0^1    E(|  (K_H^*Du_1)(s)|^q)  ds \right)^{\frac 1q}.
\end{align*}
The first factor converges to zero as $n\rightarrow \infty$ by property {\bf (h8)} and the second one is bounded by condition (iii).
Therefore,  (\ref{fg1}) holds.

It remains to show that $(B^H,    \dint_0^1 \phi_n(t) \delta B_t^H )$ converges in law in the space $C([0,1]) \times \mathbb{R})$ to $(B^H,  \sqrt{L} Z)$, where $Z$ is a $N(0,1)$ random variable independent of $B^H$.
This claim follows from Step 1 and the fact that
for any $t_0\in [0,1]$, $\lim_{n \rightarrow \infty} E\left[   B^H_{t_0} \dint_0^1  \phi_n(t) \delta B_t^H \right]=0$. Indeed,  
we can write
\begin{align*}
  E\left[   B^H_{t_0} \dint_0^1  \phi_n(t) \delta B_t^H \right] = \langle \mathbf{1}_{[0,t_0]} , \phi_n \rangle_{\HH} 
\end{align*}
and
\begin{align*}
| \langle \mathbf{1}_{[0,t_0]} , \phi_n \rangle_{\HH} | &
=   \left| \int_0^1  (K_H^* \phi_n)(s) (K_H^*  \mathbf{1}_{[0,t_0]})(s) ds \right| \\
& \le \| K_H^* \phi_n \|_{L^p([0,1[)} \|K_H^*  \mathbf{1}_{[0,t_0]} \| _{L^q([0,1])},
\end{align*}
where $\frac 1p + \frac 1q =1$.  Then the result follows from property {\bf (h8)} and the fact that
\[
\|K_H^*  \mathbf{1}_{[0,t_0]} \| _{L^q([0,1])} <\infty.
\]
\end{proof}

 Theorem \ref{thm7} can be applied to the example $\phi_n(t) = n^H t^n$, when $H >\frac 14$. Indeed, condition {\bf (h4)}, again with $L= H\Gamma(2H)$, holds by Lemma \ref{lem5} below. Condition {\bf (h3)} is obvious.  Property {\bf (h6)} follows from the following computations:
  \begin{align*}
 & \dint_0^1 ((1-s)^{2H-1}+s^{2H-1})\phi_n^2(s)   \, ds \\
& = n^{2H} \dint_0^1 ((1-s)^{2H-1}s^{2n}+ s^{2H-1+2n}) \, ds  \\
& = \dfrac{n^{2H}\Gamma(2H)\Gamma(2n+1)}{\Gamma(2n+2H+1)} + \dfrac{n^{2H} \Gamma(2n+2H) }{\Gamma(2n+2H+1)},
\end{align*}
which is uniformly bounded by Lemma \ref{lem1}.
 In order to show property {\bf (h7}), we write, for any $\delta \in [0,1)$,
 
 \begin{align*}
 & n^{2H} \int_0^{\delta} \left( \dint_s^{\delta}  (t^n-s^n)  (t-s)^{H-\frac 32} dt \right)^2 \, ds  \\
&=  n^{2H}   \int_0^{\delta} \left( \dint_s^{\delta}   \sum_{k=0}^{n-1}  t^ks^{n-1-k}(t-s)^{H-\frac 12} dt \right)^2 \, ds\\
& \le  n^{2H+2} \delta ^{2n-2} \int_0^{\delta} \left( \dint_s^{\delta}  (t-s)^{H-\frac 12} dt \right)^2   \ ds \\
& = Cn^{2H+2} \delta ^{2n-2}
 \end{align*}
 which converges to zero as $n\rightarrow \infty$. 
 
 To show property {\bf (h8)}, we write
 \begin{align*}
 \int_0^1 |K_H^* \phi_n|^p(s) ds  &\le
 C  n^{pH}\int_0^1  |K(1,s)|^p s^{np} ds + C n^{pH} \int_0^1 \left| \int_s^1 (t^n-s^n) (t-s)^{ H-\frac 32} dt \right|^p ds \\
 &\le C n^{pH} \int_0^1   ((1-s)^{p(H-\frac 12)} + s^{p (H-\frac 12)} ) s^{np} ds\\
 &+ C n^{pH+2}  \int_0^1 \left| \int_s^1  t^{n-1} (t-s)^{ H-\frac 12} dt \right|^p ds\\
 &=: B_{1,n} + B_{2,n}. 
  \end{align*}
 For the term $B_{1,n}$, we have
 \begin{align*}
 B_{1,n} \le C \left( \frac { n^{pH} \Gamma(p(H-\frac 12) +1) \Gamma(np+1)}{ \Gamma(p(H+\frac 12) +2 + np)}
 + \frac {n^{pH}} {p(H-\frac 12+n)+1}  \right)
 \end{align*}
 By Lemma \ref{lem1}, this term converges to zero as $n\rightarrow \infty$, provided $p< 2$.
 The same conclusion can be deduced for the term $B_{2,n}$ using Young's inequality.

 \begin{lemma}  \label{lem5}
 For any $H \in (0, 1/2)$, we have
 \[
 \lim_{n\rightarrow \infty}  n^H \|  t^n \|_{\HH} = \sqrt{H\Gamma (2H)}.
 \]
  \end{lemma}
  
  \begin{proof}
 Using the  operator $K^*_H$ and integrating by parts, we can write
 \begin{align}
n^{2H}  \|  t^n \|^2_{\HH} &= n^{2H}   \|K_H^*(t^n)\|_{L^2([0,1])}^2 \nonumber \\
&= n^{2H} \dint_0^1  \left( K_H(1,s)s^n+\dint_s^1 (t^n-s^n)\dfrac{\partial K_H}{\partial t }(t,s)\, dt \right)^2\, ds \nonumber \\
&= n^{2H}   \dint_0^1  \left( K_H(1,s)-n\dint_s^1 t^{n-1} K_H (t,s)\, dt \right)^2\, ds \nonumber \\
&= n^{2H}   \dint_0^1  K_H(1,s)^2\, ds -2n^{2H+1} \dint_0^1 \dint_s^1 t^{n-1} K_H t,s) K_H(1,s)\, dt  \, ds\, dt \nonumber \\
& + n^{2H+2}   \dint_0^1 \left( \dint_s^1 t^{n-1}K_H(t,s)\, dt \right)^2\, ds \nonumber \\
& =: A_{1,n} + A_{2,n} + A_{3,n}. \label{Var1}
\end{align}
At this step we work each term in (\ref{Var1}) separately. 
Since $$R_H(t,s)=\dint_0^{t\wedge s} K_H(t,u)K_H(s,u)\, du, $$
the first term is $A_{1,n}= n^{2H}R(1,1)=n^{2H}$.
Changing the order of integration in the second term yields
\begin{align*}
A_{2,n}& =2n^{2H+1}\dint_0^1 \dint_0^t t^{n-1} K_H (t,s) K_H(1,s)\, ds  \, dt \\
&= 2n^{2H+1} \dint_0^1 t^{n-1}R(1,t)\, dt \\
& = n^{2H+1} \dint_0^1 t^{n-1}(1+t^{2H}-(1-t)^{2H})\, dt \\
& = n^{2H}+\dfrac{n^{2H+1}}{n+2H}-\dfrac{n^{2H+1}\Gamma(n)\Gamma(2H+1)}{\Gamma(n+2H+1)}.
\end{align*}
Writing the third term as a triple integral, changing the order of integration and using Lemma \ref{lem2} gives
\begin{align*}
A_{3,n}&=  n^{2H+2}\dint_0^1 \dint_s^1 \dint_s^1 t^{n-1}K_H(t,s)u^{n-1}K_H(u,s)\, du\, dt\, ds \\
&= n^{2H+2}\! \dint_0^1 \! \dint_0^t \! \dint_0^u t^{n-1}K_H(t,s)u^{n-1}K_H(u,s)\, ds\, du\, dt \\
& \quad +  n^{2H+2}\dint_0^1 \! \dint_t^1 \! \dint_0^t \! t^{n-1}K_H(t,s)u^{n-1}K_H(u,s)\, ds\, du\, dt  \\
&= n^{2H+2}\dint_0^1 \dint_0^t t^{n-1} u^{n-1}R_H(t,u) \, du\, dt +  n^{2H+2}\dint_0^1 \dint_t^1 \ t^{n-1}u^{n-1}R_H(t,u) \, du\, dt \\
&= \frac{n^{2H+2}}{2} \dint_0^1 \dint_0^1 t^{n-1} u^{n-1}(t^{2H}+u^{2H}-\vert t-u\vert^{2H}) \, du\, dt \\
&= \dfrac{n^{2H+1}}{2(n+2H)}+\dfrac{n^{2H+1}}{2(n+2H)}-\dfrac{n^{2H+2}}{2}\dint_0^1 \dint_0^1 t^{n-1}u^{n-1}\vert t-u\vert^{2H}\, du\, dt \\
&= \dfrac{n^{2H+1}}{(n+2H)}-\dfrac{n^{2H+2}\ 2\  \Gamma(n)\, \Gamma(2H+1)}{2(2n+2H)\ \Gamma(n+1+2H)}.
\end{align*}
Thus (\ref{Var1}) simplifies to 
$$\dfrac{n^{2H+1}\Gamma(n)\Gamma(2H+1)}{\Gamma(n+2H+1)} -  \dfrac{n^{2H+2} \Gamma(n)\, \Gamma(2H+1)}{2(n+H)\ \Gamma(n+1+2H)},
$$
which, due to   Lemma \ref{lem1}, converges to
$$\Gamma(2H+1)-\dfrac{\Gamma(2H+1)}{2}=\dfrac{\Gamma(2H+1)}{2}=H\Gamma(2H).
 $$
\end{proof}

\end{document}